\documentclass{article}
\usepackage{amsmath}
\usepackage{latexsym}
\usepackage{amssymb}
\newtheorem{thm}{Theorem}[section]
\newtheorem{la}[thm]{Lemma}
\newtheorem{Defn}[thm]{Definition}
\newtheorem{Remark}[thm]{Remark}
\newtheorem{prop}[thm]{Proposition}

\newtheorem{Example}[thm]{Example}
\newtheorem{Number}[thm]{\!\!}

\newenvironment{example}{\begin{Example}\rm}{\end{Example}}
\newenvironment{rem}{\begin{Remark}\rm}{\end{Remark}}
\newenvironment{numba}{\begin{Number}\rm}{\end{Number}}
\newenvironment{proof}{{\noindent\bf Proof.}}%
                  {\nopagebreak\hspace*{\fill}$\Box$\medskip\par}

\newcommand{\cW}{{\mathcal W}}

\newcommand{\cL}{{\mathcal L}}

\newcommand{\cV}{{\mathcal V}}
\newcommand{\cS}{{\mathcal S}}
\newcommand{\ve}{\varepsilon}
\newcommand{\R}{{\mathbb R}}
\newcommand{\N}{{\mathbb N}}

\newcommand{\mto}{\mapsto}
\newcommand{\sub}{\subseteq}

\newcommand{\wb}{\overline}
\newcommand{\cg}{{\mathfrak g}}
\newcommand{\cm}{{\mathfrak m}}
\DeclareMathOperator{\id}{id}

\DeclareMathOperator{\Diff}{Diff}

\DeclareMathOperator{\Evol}{Evol}

\DeclareMathOperator{\op}{op}

\DeclareMathOperator{\ap}{ap}
\begin{document}
\begin{center}
{\bf\Large {\boldmath$L^1$}-regularity of strong ILB-Lie groups}\\[4mm]
{\bf Helge Gl\"{o}ckner and Ali Suri}
\end{center}
\begin{abstract}
\hspace*{-4.8mm}If $G$ is a Lie group modeled
on a Fr\'{e}chet space,
let $e$ be its neutral element and $\cg:=T_eG$
be its Lie algebra.
We show that every strong ILB-Lie group~$G$
is $L^1$-regular in the sense
that each $\gamma\in L^1([0,1],\cg)$
is the right logarithmic
derivative of some
absolutely continuous curve
$\eta\colon [0,1]\to G$
with $\eta(0)=e$
and the map $\Evol\colon L^1([0,1],\cg)\to C([0,1], G)$, $\gamma\mto\eta$
is smooth.
More generally, the conclusion holds
for a class of Fr\'{e}chet-Lie groups
considered by Hermas and Bedida.
Examples are given.
Notably, we obtain
$L^1$-regularity
for certain weighted diffeomorphism
groups.\vspace{2mm}
\end{abstract}
{\bf Classification:}
22E65 (primary);
34A12,
%
46E30,
%
46E40.\\[2.3mm]
%
%
%
\section{Introduction and statement of the results}
Regularity is an essential concept in the theory
of infinite-dimensional
Lie groups, and a common hypothesis
for many results
(cf.\ \cite{Mil,JLT,Nee,GaN}).
Consider a Lie group~$G$
modeled on a sequentially complete locally
convex space,
with neutral element~$e$
and Lie algebra $\cg:=T_eG$.
For $g\in G$, let $R_g\colon G\to G$, $x\mto xg$
be the right multiplication with~$g$.
We consider the natural right action
\[
TG \times G \to TG,\quad (v,g)\mto v.g:=TR_g(v)
\]
of~$G$
on its tangent bundle~$TG$.
Following~\cite{Mil}, $G$
is called \emph{regular}
if the initial value problem
\begin{equation}\label{ivp}
\dot{y}(t)=\gamma(t).y(t),\quad y(0)=e
\end{equation}
has a (necessarily unique) solution
$\eta\colon [0,1]\to G$
for each $C^\infty$-curve $\gamma\colon [0,1]\to\cg$
and the map
\[
C^\infty([0,1],\cg)\to G,\quad G,\quad \gamma\mto\eta(1)
\]
is smooth (or, equivalently, the map $\Evol\colon C^\infty([0,1],\cg)\to C([0,1],G)$,
$\gamma\mto\eta$ is smooth,
cf.\ \cite[Theorem~A]{SEM}).\footnote{We replaced left-invariant
vector fields with right-invariant vector fields,
compared to~\cite{Mil}.}
It became clear later that also strengthened
regularity properties are of interest.
If $k\in \N_0\cup\{\infty\}$
and $C^\infty$-maps can be replaced with $C^k$-maps in the preceding,
then $G$ is called \emph{$C^k$-regular}
(see~\cite{SEM}).
Let $p\in [1,\infty]$. If (\ref{ivp})
has a
(necessarily unique)
Carath\'{e}odory solution
$\eta\colon [0,1]\to G$
for each $\gamma\in L^p([0,1],\cg)$
and the map $\Evol\colon L^p([0,1],\cg)\to C([0,1],G)$,
$\gamma\mto\eta$ is smooth,
then $G$ is called $L^p$-regular
(see Definition 4.37 and Theorem~4.3.9
in~\cite{Nik}; cf.\ \cite{MEA} for the case
of Fr\'{e}chet-Lie groups).
By~\cite{SEM,MEA,Nik},
\[
\mbox{$L^1$-regularity}\;\Rightarrow\;\mbox{$L^p$-reg.}\;\Rightarrow\,
\mbox{$L^\infty$-reg.}\;\Rightarrow\;
\mbox{$C^0$-reg.}\;\Rightarrow\;\mbox{$C^k$-reg.}\;\Rightarrow\;
\mbox{regularity.}\vspace{1mm}
\]
Many of the most important examples
of Fr\'{e}chet-Lie groups
are
so-called
strong ILB-Lie groups
in the sense of~\cite{OMO, Omo}.
These
subsume Lie groups of smooth\linebreak
diffeomorphisms
of compact manifolds~\cite{O70}
and relevant subgroups
like
volume-preserving diffeomorphisms,
symplectomorphism groups,
and others
(see \cite{OMO,Omo,RaS}
and the references therein,
cf.\ also \cite{EaM}).
It is known that every strong ILB-Lie group
is $C^0$-regular.\footnote{Endow
$C^1_*([0,1],G)=\{\eta\in C^1([0,1],G)\colon\eta(0)=e\}$
and $C([0,1],G)$ with their standard
Lie group structures.
Then
$\Lambda\colon C^1([0,1],G)_*\to C([0,1],G)$, $\eta\mto\eta$
is a smooth group homomorphism.
By the Second Fundamental Theorem in \cite{OMY},
the right logarithmic derivative
\[
\delta^r\colon C^1_*([0,1],G)\to
C([0,1],\cg),\quad \eta\mto(t\mto \dot{\eta}(t).\eta(t)^{-1})
\]
is a $C^\infty$-diffeomorphism for
each strong ILB-Lie group~$G$.
As a consequence, the evolution map
$\Evol=\Lambda\circ (\delta^r)^{-1}\colon C([0,1],\cg)\to C([0,1],G)$
is smooth
and thus $G$
is $C^0$-regular.}
Our goal is to show that every
strong ILB-Lie group actually is $L^1$-regular,
and thus has the strongest
possible regularity property.
More generally, the conclusion will hold
for a class of Fr\'{e}chet-Lie groups
recently introduced by Hermas and Bedida~\cite{HaB},
which subsumes all strong ILB-Lie groups
(see \cite[pp.\,632--633]{HaB}).
Compared to strong ILB-Lie groups,
the Lie groups of Hermas-Bedida
are defined by a smaller number of axioms,
which are more easily checked in practice.
So far, only $C^\infty$-regularity
has been established for the latter groups (see \cite[Theorem~2.3]{HaB}).\\[2.3mm]
In the following, $\N:=\{1,2,\ldots\}$.
For $k\in\N\cup\{0,\infty\}$,
we use the symbol~$C^k$
for the $C^k$-maps of Keller's
$C^k_c$-theory~\cite{Kel} (going back to Bastiani~\cite{Bas}),
as in \cite{GaN,Ham,Mil,Nee}.
By contrast, $k$ times continuously
Fr\'{e}chet differentiable maps
between open subsets of normed spaces
shall be called \emph{$FC^k$-maps}.
Deferring details,
the setting
of~\cite{HaB}
can be outlined as follows.
\begin{numba}\label{thesett}
Let $G_1$ be a group and $G_1\supseteq G_2\supseteq\cdots$
be a descending sequence of subgroups of~$G_1$.
For each $n\in\N$,
let a smooth manifold structure
modeled on Banach spaces be given on~$G_n$.
Assume that
\begin{itemize}
\item[(a)]
For all $n,k\in\N$,
the map
\[
\mu_{n+k}^n\colon G_{n+k}\times G_n\to G_n,\quad (g,h)\mto gh
\]
is~$C^k$ (where $gh$ is the product of $g$ and $h$ in the group~$G_n$).
\end{itemize}
Moreover, three conditions
(b), (c), (d) are imposed (see~\ref{bcd})
ensuring that
\[
G_\infty:=\bigcap_{n\in\N}G_n\vspace{-1mm}
\]
can be given a smooth manifold
structure making it the projective limit
$G_\infty={\displaystyle\lim_{\displaystyle\leftarrow}}\,G_n$
as a smooth manifold.
The latter makes~$G_\infty$ a Fr\'{e}chet-Lie group.
\end{numba}
When we speak about \emph{Fr\'{e}chet-Lie groups
of the class considered by Hermas and Bedida},
we refer to such Lie groups~$G_\infty$.
Every strong ILB-Lie group belongs to this class
(see~\cite[pp.\,632--633]{HaB}).
Our main result
is the following.\vspace{1mm}
\begin{thm}\label{main-thm}
Every Fr\'{e}chet-Lie group
of the class considered
by Hermas and Bedida
is $L^1$-regular.
Notably,
every strong ILB-Lie group is $L^1$-regular.
\end{thm}
We mention
some examples
beyond the paradigms of strong ILB-Lie groups
already mentioned.
\begin{example}\label{exa-fix}
For each $d\in\N$
and compact, convex subset $K\sub\R^d$
with non-empty interior~$K^0$,
let
$\Diff_{\partial K}(K)$
be the Fr\'{e}chet-Lie group
of all $C^\infty$-diffeomorphisms
$\gamma\colon K\to K$
fixing the boundary $\partial K$ pointwise
(see~\cite{CVX}).
Then $\Diff_{\partial K}(K)$
belongs to the class
considered
by Hermas and Bedida,
and thus $\Diff_{\partial K}(K)$
is $L^1$-regular.
In the literature, only $C^0$-regularity had been established so far
(see~\cite{CVX}).\footnote{After we identified $\Diff_{\partial K}(K)$
as an example of our approach, an alternative
(more complicated) strategy to $L^1$-regularity of this group
was worked out in~\cite{Cha}, under supervision of the first author.
The latter can also be deduced from the $L^1$-regularity of
$\Diff_L(\R^d)$ (the Lie group of $C^\infty$-diffeomorphisms
$\psi\colon \R^d\to\R^d$ such that $\psi(x)=x$ for all $x\in \R^d\setminus L$)
for a compact subset $L\sub\R^d$ with $K$ in its interior
(as established in \cite{MEA}).}
\end{example}
\begin{example}\label{exa-polytope}
For each $d\in\N$ and each convex polytope
$K\sub\R^d$ with $K^0\not=\emptyset$,
the group $\Diff(K)$
of all $C^\infty$-diffeomorphisms
of~$K$
is a Fr\'{e}chet-Lie group~\cite{POL}.
It belongs to the class
considered by Hermas and Bedida,
whence it is $L^1$-regular.
\end{example}
\begin{example}\label{exa-sun-topa}
For $d\in\N$,
let
$\Diff_{\ap}(\R^d)$
be the Fr\'{e}chet-Lie group
of almost periodic smooth
diffeomorphisms of~$\R^d$
constructed by Sun and Topalov~\cite{SaT}.
Then $\Diff_{\ap}(\R^d)$
belongs to the class
considered by Hermas and Bedida,
whence $\Diff_{\ap}(\R^d)$
is $L^1$-regular.
So far, only the existence of an
exponential map had been
established in~\cite{SaT}.
\end{example}
We shall deduce Theorem~\ref{main-thm}
from a
technical result
(Proposition~\ref{technical})
dealing with
Fr\'{e}chet-Lie groups~$G$
which are an open subset
of an intersection
$\bigcap_{n\in\N}M_n$
of smooth Banach manifolds~$M_n$
which need not be groups.
The proposition applies,
in particular, to certain weighted
diffeomorphisms groups.
To formulate the application,
let $(X,\|\cdot\|_X)$
be a Banach space
and $\Diff(X)$
be the group of all $C^\infty$-diffeomorphisms
$\gamma\colon X\to X$.
If $U\sub X$ is an open subset,
$(Y,\|\cdot\|_Y)$ a Banach space,
$\ell\in \N\cup\{0,\infty\}$
and
$\cW$
a set of functions
$w\colon U\to\R\cup\{\infty\}$
such that the constant
function $1\colon U\to \R$
is in~$\cW$, one can define a
certain locally convex space
$FC^\ell_\cW(U,Y)$
of weighted $FC^\ell$-functions $\gamma\colon U\to Y$
(see \cite{Wal}).
If $\ell=\infty$ or $X$ has finite dimension,
we relax notation and abbreviate
$C^\ell_\cW(U,Y):=FC^\infty_\cW(U,Y)$.
Walter~\cite{Wal} turned
\[
\Diff_\cW(X):=
\left\{\gamma\in \Diff(X)\colon \gamma-\id_X,\,
\gamma^{-1}-\id_X\in C^\infty_\cW(X,X)\right\}
\]
into a regular Lie group modeled on $C^\infty_\cW(X,X)$.
If $\cW$ is countable, then $C^\infty_\cW(X,X)$
is a Fr\'{e}chet space and hence
$\Diff_\cW(X)$ a Fr\'{e}chet-Lie group.
One can show
that Proposition~\ref{technical}
applies to $\Diff_\cW(X)$ in this case
and deduce:
\begin{prop}\label{main-2}
For every Banach space $(X,\|\cdot\|_X)$
and countable set $\cW$ of functions
$w\colon X\to\R\cup\{\infty\}$ such that $1\in\cW$,
the weighted diffeomorphism group
$\Diff_\cW(X)$ is $L^1$-regular.
\end{prop}
\begin{example}\label{exa-bdd}
For any Banach space $(X,\|\cdot\|_X)$,
taking $\cW:=\{1\}$
we obtain the Fr\'{e}chet-Lie group
$\Diff_{\{1\}}(X)$
of all $C^\infty$-diffeomorphisms
$\gamma\colon X\to X$
such that $\gamma-\id_X$
and $\gamma^{-1}-\id_X$
are bounded functions
whose $k$th derivatives\linebreak
$X\to \cL^k(X,X)$
are bounded for all $k\in\N$
(where $\cL^k(X,X)$ is the Banach space of all continuous
$k$-linear maps $X^k\to X$).
By Proposition~\ref{main-2},
$\Diff_{\{1\}}(X)$ is $L^1$-regular.
\end{example}
\begin{example}\label{exa-schwartz}
Take $X=\R^d$
for some $d\in\N$.
If $\cW$ is the set of all polynomial
functions $\R^d\to\R$,
then $C^\infty_\cW(\R^d,\R^d)\cong \cS(\R^d)^d$
is the Schwartz space of $\R^d$-valued
rapidly decreasing smooth functions
and $\Diff_\cW(\R^d)$
the corresponding Lie group
of rapidly decreasing diffeomorphisms,
denoted $\Diff_\cS(\R^d)$ in~\cite{Wal}.\footnote{The Lie group structure on
$\Diff_\cS(\R)$
was already established in~\cite{Mic}. 
Using convenient differential calculus~\cite{KaM},
$\Diff_\cS(\R^d)$
and $\Diff_{\{1\}}(\R^d)$ were
discussed in~\cite{MaM}.
Compare also~\cite{Gol}.}
If $\|\cdot\|_2$ is the euclidean
norm on~$\R^d$
and $\cV$ the countable set of
polynomials
\[
\R^d\to\R,\quad x\mto (1+\|x\|_2^2)^m
\]
for $m\in\N\cup\{0\}$,
then $\Diff_\cW(\R^d)=\Diff_\cV(\R^d)$.
Thus $\Diff_\cS(\R^d)=\Diff_\cV(\R^d)$
is $L^1$-regular, by Proposition~\ref{main-2}.
\end{example}
If $G$ is a Banach-Lie group
and $M$ a compact smooth manifold
(possibly with boundary),
then the Lie group
$C^\infty(M,G)$
is $L^1$-regular.
We could prove this using Theorem~\ref{main-thm},
but the result is known
(see \cite[Remark~7.12]{MEA})
and follows more easily from
the fact that $C^\infty(M,G)=\bigcap_{n\in\N}C^n(M,G)$
is an intersection of Banach-Lie groups
which admits a so-called projective limit chart,
a situation covered by \cite[Proposition 7.8]{MEA}.
The following related examples are new.\\[2.3mm]
Let $G$ be a Banach-Lie group with Lie algebra
$\cg$ and exponential function $\exp_G\colon \cg\to G$.
If $(X,\|\cdot\|_X)$ is
a Banach space, $U\sub X$ an open subset,
$\ell\in \N\cup\{0,\infty\}$ and
$\cW$ a set of functions
$w\colon U\to \R\cup\{\infty\}$
such that $1\in \cW$,
then the subgroup
\[
FC^\ell_\cW(U,G)
\]
of $G^U$ generated by $\{\exp_G\circ\, \gamma\colon \gamma\in FC^\ell_\cW(U,\cg)\}$
admits a unique Lie group
structure modeled on $FC^\ell_\cW(U,\cg)$
making the map
\[
FC^\ell_\cW(U,\cg)\to FC^\ell_\cW(U,G),\quad
\gamma\mto\exp_G\circ\, \gamma
\]
a local $C^\infty$-diffeomorphism at~$0$
(cf.\ \cite[Definition 6.1.8]{Wal}).
The Lie groups obtained are $C^\infty$-regular~\cite[Proposition 6.1.14]{Wal}.
Proposition~7.8 in \cite{MEA} implies:
\begin{prop}\label{weight-map-reg}
If $\cW$ is countable,
then $FC^\ell_\cW(U,G)$
is $L^1$-regular.
\end{prop}
\begin{example}\label{exa-again}
Using $\cW$ and $\cV$ as in Example~\ref{exa-schwartz},
we see that the Lie group
\[
\cS(\R^d,G):=FC^\infty_\cW(\R^d,G)=FC^\infty_\cV(\R^d,G)
\]
of rapidly decreasing $G$-valued mappings on $\R^d$
is $L^1$-regular, for each Banach-Lie group~$G$.
\end{example}
More general types of Lie groups of rapidly decreasing mappings
were considered in \cite{BCR}
and \cite[\S\,6.2]{Wal}. The following observation
uses notation as in~\cite{Wal}.
\begin{prop}\label{BCR-prop}
For each $d\in\N$, countable
non-empty set $\cW$ of BCR-weights
$w\colon \R^d\to[1,\infty]$
in the sense of {\rm\cite[Definition 6.2.18]{Wal}}
and each Banach-Lie group~$G$,
the Lie group $\cS(\R^d,G;\cW)$
of $\cW$-rapidly decreasing
$G$-valued smooth mappings on~$\R^d$
is $L^1$-regular.
\end{prop}
We mention that $C^k$-regularity
has been
discussed for
weighted mapping groups with values in
a $C^k$-regular Lie group~$G$
which need not be a Banach-Lie group
for $k\in\N\cup\{0,\infty\}$,
under suitable hypotheses for the set $\cW$
of weights~\cite{Ni0}.\\[2.3mm]
%
Independently, related studies (without direct overlap)
were performed recently by
Pierron and Trouv\'{e}~\cite{PaT}. To explain them, consider a Banach
half-Lie group~$G$ in the sense of~\cite{BHM}.
Thus $G$ is a topological group,
endowed with a structure of smooth Banach manifold
making the right translation $R_g\colon G\to G$, $x\mto xg$
a smooth map for each $g\in G$. For $k\in \N$, let $G_k$
be the set of all elements $g\in G$
for which the left translations $x\mto gx$ and $x\mto g^{-1}x$
are $FC^k$. If $G$ admits a right-invariant local addition,
then also each $G_k$ is a Banach half-Lie group~\cite{BHM}
and the existence of evolution maps for $L^1$-vector fields
can be studied on the level of the manifolds~$G_k$
(cf.\ \cite{PaT}). For complementary results concerning regularity properties
of half-Lie groups, cf.\ also \cite{Pin}.
\section{Preliminaries and notation}
All vector spaces we consider are
vector spaces over the field~$\R$
of real numbers.
If $(E,\|\cdot\|_E)$
is a normed space,
write
$\wb{B}^E_r(x):=\{y\in E\colon \|y-x\|_E\leq r\}$
for the closed ball or radius $r\geq 0$ around $x\in E$.
If $(E,\|\cdot\|_E)$
and $(F,\|\cdot\|_F)$
are normed spaces,
we let $(\cL(E,F),\|\cdot\|_{\op})$
be the normed space of all continuous
linear maps $\alpha\colon E\to F$,
where $\|\cdot\|_{\infty}$
is the operator norm.
Likewise, $(\cL^k(E,F),\|\cdot\|_{\op})$
denotes the normed space of all continuous $k$-linear maps $\alpha\colon E^k\to F$
with
\[
\|\alpha\|_{\op}:=\sup\{\|\alpha(x_1,\ldots, x_k)\|_F\colon x_1,\ldots, x_k
\in
\wb{B}^E_1(0)\},
\]
if $k\in\N$.
We are working in Bastiani's setting of infinite-dimensional
calculus~\cite{Bas}, also known as Keller's $C^k_c$-theory~\cite{Kel}.
For
concepts and notation concerning
infinite-dimensional calculus and
manifolds modeled on locally convex spaces,
the reader is referred to \cite{GaN, Nee, Sme}
(cf.\ also \cite{BCR,Ham,Mil}).
Notation concerning vector-valued $\cL^p$-functions,
absolutely continuous functions,
and Carath\'{e}odory solutions to differential equations
is as in \cite{Nik,GaH, MEA}
(cf.\ also \cite{Sch}).
\section{The setting of Hermas and Bedida}
We consider Fr\'{e}chet--Lie groups in a setting introduced
by Hermas and Bedida~\cite{HaB}.
Slightly reformulated in an equivalent form,
the setting can be described as follows.
\begin{numba}\label{bcd}
Let $G_1$ be a group, $G_1\supseteq G_2\supseteq\cdots$
be subgroups, and $G_\infty:=\bigcap_{n\in \N} G_n$.
For each $n\in\N$,
let a smooth manifold structure
modeled on Banach spaces be given on~$G_n$.
Let $\{\phi_j\colon j\in J_1\}$
be a $C^\infty$-atlas for~$G_1$,
consisting of charts $\phi_j\colon U_j\to V_j$
from open non-empty subsets $U_j\sub G_1$
onto open subsets~$V_j$ of Banach spaces~$E_{1,j}$.
For all $(n,j)\in \N\times J_1$ with $n\geq 2$
such that $G_n\cap U_j\not=\emptyset$,
let $E_{n,j}$ be a Banach space.
For integers $n\geq 2$, abbreviate
\[
J_n:=\{j\in J_1\colon G_n\cap U_j\not=\emptyset\};
\]
let
\[
J:=\{j\in J_1\colon G_\infty\cap U_j\not=\emptyset\}.
\]
The following assumptions (a)--(d) are made:
\begin{itemize}
\item[(a)]
For all $n,k\in\N$,
the map
\[
\mu_{n+k}^n\colon G_{n+k}\times G_n\to G_n,\quad (g,h)\mto gh
\]
is~$C^k$ (where $gh$ is the product of $g$ and $h$ in the group~$G_n$).
\item[(b)]
$E_{n+1,j}$ is vector subspace of
$E_{n,j}$ for all $n\in\N$ and $j\in J_{n+1}$.
Moreover, the inclusion map $E_{n+1,j}\to E_{n,j}$
(which is linear) is continuous.
\item[(c)]
For all $n\in\N$ and $j\in J_n$,
the intersection $G_n\cap U_j$
is open in~$G_n$,
the image $V_{n,j}:=\phi_j(G_n\cap U_j)$
is open in~$E_{n,j}$, and
\[
\phi_{n,j}:=\phi_j|_{G_n\cap \, U_j}\colon G_n\cap U_j\to V_{n,j}
\]
is a chart of~$G_n$.
\item[(d)]
For all $j\in J$,
the set $W_j:=\phi_j(G_\infty\cap U_j)$
is open in $E_j:=\bigcap_{n\in\N}E_{n,j}$,
endowed with the projective limit topology.
\end{itemize}
\end{numba}
Give $G_\infty$ the initial topology with
respect to the inclusion maps $\lambda_n\colon G_\infty\to G_n$
for $n\in\N$, making it the projective limit ${\displaystyle\lim_{{\textstyle\leftarrow}}}\,G_n$\vspace{-.7mm}
as a topological space.
Hermas and Bedida~\cite{HaB} showed
that, for $j\in J$, the maps
\[
\psi_j\colon G_\infty\cap U_j\to W_j,\quad
x\mto\phi_j(x)
\]
form a $C^\infty$-atlas for~$G_\infty$
making it a Fr\'{e}chet-Lie group
which is $C^\infty$-regular.
\begin{rem}\label{more-details}
\begin{itemize}
\item[(a)]
If $j\in J$ and
$Q$
is an open subset of~$G_1$,
then $U_j\cap Q$ is open in $U_j$,
whence $\phi_j(U_j\cap Q)$ is open in $V_j$
and hence in $E_{1,j}$.
As a consequence,
\[
\phi_j(G_\infty\cap U_j\cap Q)
=\phi_j(G_\infty\cap U_j)\cap \phi_j(U_j\cap Q)
\]
is open in $E_j$.
Thus (d) is equivalent
to the
stronger-looking condition~(iii)
of \cite[p.\,631]{HaB}.
\item[(b)]
While we only require
that the $\mu_{n+k}^n$
be Keller $C^k_c$-maps,
they are assumed $k$ times
continuously Fr\'{e}chet
differentiable maps in~\cite{HaB}
(i.e., $FC^k$-maps).
This is inessential:
If we define $H_n:=G_{2n-1}$
in~\ref{bcd}
for $n\in\N$,
then $G_\infty=\bigcap_{n\in\N}H_n$.
Moreover, the topology on $E_j$
and the maps $\psi_j:=\phi_j|_{G_\infty\cap U_j}\colon G_\infty\cap U_j\to W_j$
are unchanged if we replace $G_n$ with $H_n$, and
also the projective limit topologies
on $\bigcap_{n\in\N}G_n$
and $\bigcap_{n\in\N}H_n=\bigcap_{n\in\N}G_{2n-1}$
coincide.
As the product map
$H_{n+k}\times H_n\to H_n$
equals $\mu_{2n+2k-1}^{2n-1}$,
it is $C^{2k}$, hence $C^{k+1}$
and hence $k$ times continuously
Fr\'{e}chet differentiable
(see, e.g., \cite{Kel,GaN,Wal}).
\end{itemize}
\end{rem}
\begin{numba}
Condition~(d)
is automatic if
\[
\phi_j(G_n\cap U_j)=E_{n,j}\cap V_j
\]
for all $j\in J$
(this variant is called condition~(iv) in \cite[p.\,631]{HaB}).
\end{numba}
\section{Lemmas concerning Carath\'{e}odory solutions}
It is useful to formulate some quantitative results concerning
Carath\'{e}odory solutions to ordinary differential equations (ODE).
We recall a basic estimate of calculus.

\begin{la}\label{lip-para}
Let $(E,\|\cdot\|_E)$,
$(F,\|\cdot\|_F)$,
and $(Z,\|\cdot\|_Z)$
be normed spaces, $U\sub E$
be an open subset, and $f\colon U\times Z\to F$
be a $C^2$-map such that $f(y,\cdot)\colon Z\to F$, $z\mto f(y,z)$
is linear for each $y\in U$.
Let $y_0\in U$.
Then there exists an open $y_0$-neighbourhood
$U_0\sub U$ and $L\in [0,\infty[$
such that
\[
(\forall y_1,y_2\in U_0)(\forall z\in Z)\;\;
\|f(y_2,z)-f(y_1,z)\|_F\,\leq\, L\hspace*{.2mm}\|y_2-y_1\|_E\hspace*{.2mm}\|z\|_Z.
\]
\end{la}
\begin{proof}
The map $f^\vee\colon U\to\cL(Z,F)$, $y\mto f(y,\cdot)$
is $C^1$ to $(\cL(Z,F),\|\cdot\|_{\op})$, by \cite[Proposition~1]{ASP}.
As a consequence, $f^\vee$ is locally
Lipschitz (cf.\ \cite[Exercise~1.5.4]{GaN}).
Thus, there exists an open $y_0$-neighbourhood
$U_0\sub U$ and $L\geq 0$
such that $\|f^\vee(y_2)-f^\vee(y_1)\|_{\op}\leq L\,\|y_2-y_1\|_E$
for all $y_1,y_2\in U_0$.
\end{proof}
The following quantitative variant of
the Carath\'{e}odory Existence Theorem
is useful
for us (cf.\ \cite[30.9]{Sch}
for other versions).
\begin{la}\label{carath-ex}
Let $(E,\|\cdot\|_E)$
and $(Z,\|\cdot\|_Z)$
be Banach spaces, $y_0\in E$,
$R>0$ and
$f\colon \wb{B}^E_R(y_0)\times Z\to E$
be a continuous map such that $f(y,\cdot)\colon Z\to E$
is linear for each $y\in \wb{B}^E_R(y_0)$
and there exists $L\in [0,\infty[$
such that
\begin{equation}\label{integral-lip}
\|f(y_2,z)-f(y_1,z)\|_E\,\leq\, L\hspace*{.1mm}\|y_2-y_1\|_E\|z\|_Z
\end{equation}
for all $y_1,y_2\in \wb{B}^E_R(y_0)$ and $z\in Z$.
Then the following holds
for all real numbers $a<b$:
\begin{itemize}
\item[\rm(a)]
If $\gamma\colon [a,b]\to Z$ is an $\cL^1$-map,
then the ODE $y'(t)=f(y(t),\gamma(t))$
satisfies
local
uniqueness of Carath\'{e}odory
solutions.
\item[\rm(b)]
If, moreover,
$L\hspace*{.2mm}\|\gamma\|_{\cL^1}<1$
and
\[
\sup\left\{\|f(y,\cdot)\|_{\op}\colon y\in \wb{B}^E_R(y_0)\right\}
 \|\gamma\|_{\cL^1}
\;\leq\; R,
\]
then, for each $t_0\in [a,b]$, the initial value problem
\[
y'(t)=f(y(t),\gamma(t)),\qquad
y(t_0)=y_0
\]
has a Carath\'{e}odory
solution $\eta\colon [a,b]\to \wb{B}^E_R(y_0)$
on all of $[a,b]$.
\end{itemize}
\end{la}
\begin{proof}
(a) In view of (\ref{integral-lip}),
local uniqueness is guaranteed by \cite[30.9]{Sch}
or \cite[Proposition~10.5]{MEA}.

(b) Write $y_0$ also for the constant function
$[a,b]\to E$, $t\mto y_0$.
Then
\[
X:=\{\eta\in C([a,b],E)\colon\mbox{$\eta(t_0)=y_0$ and $\|\eta-y_0\|_\infty\leq R$}\}
\]
is a non-empty, closed subset of
the Banach space $C([a,b],E)$. For each $\eta\in X$,
the function $[a,b]\to E$, $t\mto f(\eta(t),\gamma(t))$
is $\cL^1$ by \cite[Lemma~2.1]{MEA}, whence
\[
\Phi(\eta)\colon [a,b]\to E,\quad
t\mto y_0+\int_{t_0}^t f(\eta(s),\gamma(s))\, ds
\]
is an absolutely continuous function
with $\Phi(\eta)(t_0)=y_0$.
Then $\|\Phi(\eta)(t)-y_0\|_E\leq R$
for all $t\in [a,b]$, whence $\Phi(\eta)\in X$.
In fact,
\begin{eqnarray*}
\|\Phi(\eta)(t)-y_0\|_E
& \leq & \int_{t_0}^t \|f(\eta(s),\gamma(s))\|_E\,ds\\
&\leq & \sup\big\{\|f(y,\cdot)\|_{\op}\colon y\in\wb{B}^E_R(y_0)\big\}
\, \|\gamma\|_{\cL^1}
\leq R
\end{eqnarray*}
for $t\in [t_0,b]$ (and likewise
for $t\in [a,t_0]$).
For all $\eta,\zeta\in X$, we have
\[
\|\Phi(\eta)(t)-\Phi(\zeta)(t)\|_E
\leq\int_{t_0}^t
\|f(\eta(s),\gamma(s))-f(\zeta(s),\gamma(s))\|_E\, ds
\leq L\hspace*{.1mm}\|\gamma\|_{\cL^1} \hspace*{.1mm}
\|\eta-\zeta\|_\infty
\]
for all $t\in [t_0,b]$ (and likewise
for $t\in [a,t_0]$). Hence,
$\Phi\colon X\to X$ is a contraction
with Lipschitz constant $L\|\gamma\|_{\cL^1}<1$.
By Banach's Fixed Point Theorem,
$\Phi$ has a unique fixed point $\eta$.
Then $\eta=\Phi(\eta)$
is absolutely continuous;
by construction, $\eta$
solves the desired initial value problem.
\end{proof}
\begin{la}\label{mfd-case}
Let $(Z,\|\cdot\|_Z)$
be a Banach space, $M$ be a smooth manifold
modeled on a Banach space
and $f\colon M\times Z \to TM$
be a $C^2$-map such that $f(p,z)\in T_pM$
for all $p\in M$ and $z\in Z$,
and the map $f(p,\cdot)\colon Z\to T_pM$
is linear for each $p\in M$.
For all real numbers $a<b$,
the following then holds.
\begin{itemize}
\item[\rm(a)]
For each $\cL^1$-function $\gamma\colon [a,b]\to Z$,
the ODE $\dot{y}(t)=f(y(t),\gamma(t))$
satisfies local existence and local
uniqueness of Carath\'{e}odory solutions.
\item[\rm(b)]
For all $y_0\in M$
and $t_0\in [a,b]$,
there exists $\ve>0$
such that, for each $\cL^1$-function
$\gamma\colon [a,b]\to Z$ with $\|\gamma\|_{\cL^1}\leq \ve$,
the initial value problem
\[
\dot{y}(t)=f(y(t),\gamma(t)),\quad y(t_0)=y_0
\]
has a Carath\'{e}odory
solution $\eta\colon [a,b]\to M$
on all of~$[a,b]$.
\end{itemize}
\end{la}
\begin{proof}
(a)
Given $y_0\in M$,
there exists a chart $\phi\colon U\to V\sub E$
with $y_0\in U$,
for some Banach space
$(E,\|\cdot\|_E)$.
Then
\[
g\colon V\times Z\to E,\quad (x,z)\mto d\phi (f(\phi^{-1}(x),z))
\]
is a $C^2$-map and $g(x,\cdot)\colon Z\to E$
is linear for each $x\in V$.
Let $x_0:=\phi(y_0)$.
By Lemma~\ref{lip-para},
after shrinking the open $x_0$-neighbourhood
$V$ (and $U=\phi^{-1}(V)$),
we may assume that, for some $L\in[0,\infty[$,
\[
\|g(x_2,z)-g(x_1,z)\|_E\leq L\hspace*{.1mm}\|x_2-x_1\|_E\|z\|_Z
\]
for all $x_1,x_2\in V$ and $z\in Z$.
Hence, if
$a<b$ are real numbers
and $\gamma\colon [a,b]\to Z$
is an $\cL^1$-function,
then
\[
x'(t)=g(x(t),\gamma(t))
\]
satisfies local uniqueness
of Carath\'{e}odory solutions, by Lemma~\ref{carath-ex}\,(a).
As a consequence, $\dot{y}(t)=(f|_{U\times Z})(y(t),\gamma(t))$
satisfies local uniqueness of Carath\'{e}odory
solutions
(the vector fields
$y\mto f(y,\gamma(t))$ on~$U$
and $x\mto (x,g(x,\gamma(t)))$ on~$V$
being $\phi$-related for all $t\in [a,b]$).
As $y_0$ was arbitrary and the corresponding open sets~$U$
cover~$M$,
the ODE $\dot{y}(t)=f(y(t),\gamma(t))$
satisfies local uniqueness as well.
For some $R>0$, we have
$\wb{B}^E_R(x_0)\sub V$.
The mapping
\[
g^\vee\colon V\to (\cL(Z,E),\|\cdot\|_{\op}),\quad x\mto g(x,\cdot)
\]
being continuous by \cite[Proposition~1]{ASP},
after shrinking~$R$ we may assume that
\[
S:=\sup\{\|g(x,\cdot)\|_{\op}\colon x\in \wb{B}^E_R(x_0)\}<\infty.
\]
Choose $\ve>0$ so small that
\[
L\ve<1\quad\mbox{and}\quad S\ve \leq R.
\]
For each
$t_0\in [a,b]$,
there exist $\alpha\in [a,t_0]$
and $\beta\in [t_0,b]$
such that $[\alpha,\beta]$
is a neighbourhood of~$t_0$ in $[a,b]$
and $\|\gamma|_{[\alpha,\beta]}\|_{\cL^1}\leq\ve$.
By Lemma~\ref{carath-ex}\,(b), the initial value problem
\[
x'(t)=g(x(t),\gamma(t)),\quad x(t_0)=x_0
\]
has a Carath\'{e}odory solution $\eta\colon [\alpha,\beta]\to \wb{B}^E_R(x_0)$.
Then $\phi^{-1}\circ\eta$ solves $\dot{y}(t)=f(y(t),\gamma(t))$,
$y(t_0)=y_0$.

(b) Let $\ve$ be as in the proof of~(a).
If $\|\gamma\|_{\cL^1}\leq\ve$,
we can take $\alpha:=a$ and $\beta:=b$
in the proof of~(a).
\end{proof}
\section{Further generalizations}
To enable the study of further examples like certain
weighted diffeomorphism groups,
we now introduce a more general setting than the one
of Hermas and Bedida.
We consider the following situation.
\begin{numba}\label{new-sett}
Let $G$ be a Fr\'{e}chet-Lie group
with group multiplication
$\mu_G\colon G\times G\to G$
and neutral element~$e$.
Assume that~$G$
is an open submanifold
of a smooth manifold~$M_\infty$
such that
\[
M_\infty=\bigcap_{n\in\N}M_n
\]
for a sequence
$M_1\supseteq M_2\supseteq\cdots$
of smooth Banach manifolds,
with smooth inclusion maps
$\lambda_m^n\colon M_m\to M_n$
for all $m\geq n$ in~$\N$.
We assume that a map
\[
\mu_2^1\colon M_2\times M_1\to M_1,\quad (x,y)\mto xy
\]
is given with the following properties:
\begin{itemize}
\item[(a1)]
$\mu_2^1(x,y)=\mu_G(x,y)$
for all $x,y\in G$;
\item[(a2)]
$\mu_2^1(M_{n+k}\times M_n)\sub M_n$
for all $n,k\in \N$ and the restriction
\[
\mu_{n+k}^n\colon M_{n+k}\times M_n\to M_n,\quad (x,y)\mto xy
\]
is~$C^k$;
\item[(a3)]
$ex=x$ for all $x\in M_1$;
and
\item[(a4)]
$(xy)z=x(yz)$ holds for all $\ell>m >n$ in~$\N$ and
$(x,y,z)\in M_\ell\times M_m\times M_n$.
\end{itemize}
Let $\{\phi_j\colon j\in J_1\}$
be a $C^\infty$-atlas for~$M_1$,
consisting of charts $\phi_j\colon U_j\to V_j$
from open non-empty subsets $U_j\sub M_1$
onto open subsets~$V_j$ of Banach spaces~$E_{1,j}$.
For all $(n,j)\in \N\times J_1$
with $n\geq 2$ and $M_n\cap U_j\not=\emptyset$,
let $E_{n,j}$ be a Banach space.
Define
\[
J_n:=\{j\in J_1\colon M_n\cap U_j\not=\emptyset\}
\]
for $n\geq 2$ and
\[
J:=\{j\in J_1\colon M_\infty\cap U_j\not=\emptyset\}.
\]
We assume that (b), (c), and (d) are satisfied:
\begin{itemize}
\item[(b)]
$E_{n+1,j}$ is vector subspace of
$E_{n,j}$ for all $n\in\N$ and $j\in J_{n+1}$.
Moreover, the inclusion map $E_{n+1,j}\to E_{n,j}$
(which is linear) is continuous.
\item[(c)]
For all $n\in\N$ and $j\in J_n$,
the intersection $M_n\cap U_j$
is open in~$M_n$,
the image $V_{n,j}:=\phi_j(M_n\cap U_j)$
is an open subset of~$E_{n,j}$, and
\[
\phi_{n,j}:=\phi_j|_{M_n\cap \, U_j}\colon M_n\cap U_j\to V_{n,j}
\]
is a chart of~$M_n$.
\item[(d)]
For all $j\in J$,
the image $W_j:=\phi_j(M_\infty\cap U_j)$
is an open subset of $E_j:=\bigcap_{n\in\N}E_{n,j}$,
endowed with the projective limit topology.
Moreover, the map $M_\infty\cap U_j\to W_j$, $x\mto \phi_j(x)$
is a chart for $M_\infty$.
\end{itemize}
\end{numba}
\begin{prop}\label{technical}
In the situation of {\rm\ref{new-sett}},
the Fr\'{e}chet-Lie group~$G$
is $L^1$-regular.
\end{prop}
\begin{rem}\label{multiple}
For all $n\in\N$,
we have
\[
x_k\cdots x_2x_1:=x_k(x_{k-1}\cdots x_1)\in M_n
\]
for all $x_1,\ldots, x_k\in M_{n+1}$,
by induction on $k\in\N$ (using~(a2)).
\end{rem}
The following lemma can be shown by standard arguments.
\begin{la}\label{PL-lemma}
In the situation of Proposition~{\rm\ref{technical}},
the following holds:
\begin{itemize}
\item[\rm(a)]
For each $n\in \N$, the inclusion map
$\lambda_n\colon M_\infty\to M_n$
is smooth.
\item[\rm(b)]
$M_\infty$, together with the
inclusion maps $\lambda_n\colon M_\infty\to M_n$,
is the projective limit of
$((M_n)_{n\in \N},(\lambda_m^n)_{n\leq m})$
in the category of sets, the category of topological spaces,
and the category of smooth manifolds
modeled on Fr\'{e}chet spaces.
\item[\rm(c)]
For all real numbers $a<b$, a map
$\eta\colon [a,b]\to M_\infty$
is absolutely continuous
if and only if $\lambda_n\circ \eta\colon [a,b]\to M_n$
is absolutely continuous
for all $n\in\N$.
\item[\rm(d)]
If $a<b$ are real numbers,
$\eta\colon [a,b]\to M_\infty$
is continuous,
$t\in [a,b]$ and $v\in T_{\eta(t)}(M_\infty)$,
then $\dot{\eta}(t)$ exists and equals~$v$
if and only if $(\lambda_n\circ\eta)^{{\boldsymbol \cdot}}(t)$
exists and equals $T\lambda_n(v)$
for all $n\in\N$.
\end{itemize}
\end{la}
\section{Proof of Proposition~\ref{technical}
and Theorem~\ref{main-thm}}
It suffices to prove the proposition, Theorem~\ref{main-thm}
being a special case.
\begin{numba}\label{the-k}
Let $P$ be an open identity neighbourhood in~$G$.
Then $P$ is an open $e$-neighbourhood in~$M_\infty$.
Since $M_\infty={\displaystyle\lim_{\textstyle \leftarrow}}\, M_n$
as a topological space, there exists $\ell\in \N$
and an open $e$-neighbourhood $Q$ in~$M_{\ell+1}$
such that
\[
M_\infty\cap Q\, \sub \, P.
\]
\end{numba}
\begin{numba}
For all $n< m$ in $\N$ and $g\in M_n$, the map
\[
R^{n,m}_g:=\mu_m^n(\cdot,g)\colon M_m\to M_n,\quad x\mto xg
\]
is $C^{m-n}$. For all $n<m<k $ in $\N$ and $g\in M_n$, we have
\begin{equation}\label{triv-1}
R^{n,k}_g=R^{n,m}_g\circ \lambda_k^m
\end{equation}
and, for $h\in M_m$,
\begin{equation}\label{triv-2}
R^{n,m}_g\circ R^{m,k}_h=R^{n,k}_{hg},
\end{equation}
using (a4) in~\ref{new-sett}.
\end{numba}
%
For $n\in\N$, let $\|\cdot\|_n$ be a norm on $\cm_n:=T_eM_n$
defining its topology.
\begin{numba}\label{the-eps}
For $n\in\N$, the mapping $\mu_{n+3}^n\colon M_{n+3}\times M_n\to M_n$
is $C^3$. Identifying $T(M_{n+3}\times M_n)$
with $TM_{n+3}\times TM_n$
as usual, we obtain a $C^2$-map
\[
T\mu_{n+3}^n\colon T M_{n+3}\times TM_n\to TM_n.
\]
The zero-section
$M_n\to TM_n$, $x\mto 0_x\in T_x(M_n)$
is a smooth map.
Hence
\[
f_n\colon M_n\times \cm_{n+3}\to T M_n,\quad (x,v)\mto T\mu_{n+3}^n(v,0_x)
=TR_x^{n,n+3}(v)
\]
is a $C^2$-map. By construction,
$f_n(x,\cdot)\colon \cm_{n+3}\to T_x (M_n)$
is linear for each $x\in M_n$.
By Lemma~\ref{mfd-case}, the ODE
\[
\dot{y}(t)=f_n(y(t),\gamma(t))
\]
satisfies local existence and local uniqueness
of Carath\'{e}odory solutions
for all real numbers $a<b$ and all
$\gamma\in \cL^1([a,b],\cm_{n+3})$.
Moreover, there exists $\ve_n>0$
such that, for each $\gamma\in\cL^1([0,1],\cm_{n+3})$
with $\|\gamma\|_{\cL^1}\leq \ve_n$,
the initial value problem
\[
\dot{y}(t)=f_n(y(t)),\gamma(t)),\quad y(0)=e
\]
has a Carath\'{e}odory
solution $s_n(\gamma)\colon [0,1]\to M_n$
defined on all of $[0,1]$.
If $n=\ell+1$ (where $\ell$ is as in \ref{the-k}),
we may assume that each $s_n(\gamma)$
only takes values in~$Q$.
To see this, apply Lemma~\ref{mfd-case}
to the restriction of $f_n$ to a map $Q\times \cm_{n+3}\to TQ$.
\end{numba}
\begin{numba}\label{disc-zeta}
Let $a<b$ be real numbers, $n\in \N$,
$\gamma\in \cL^1([a,b],\cm_{n+4})$,
and $\eta\colon [a,b]\to M_{n+1}$
be a Carath\'{e}odory solution to
$\dot{y}(t)=f_{n+1}(y(t),\gamma(t))$.
For each $x\in M_n$, the function
\[
\zeta\colon [a,b]\to M_n,\quad t\mto \eta(t)x
\]
then is a Carath\'{e}odory
solution to $\dot{y}(t)=f_n(y(t),T_e(\lambda_{n+4}^{n+3})(\gamma(t)))$.
In fact, $\zeta=R^{n,n+1}_x\circ\eta$
is absolutely continuous
(see \cite[Lemma~3.18\,(a)]{MEA}).
For each $t\in [a,b]$
such that $\dot{\eta}(t)$
exists and equals $f_{n+1}(\eta(t),\gamma(t))$
(which is the case for almost all~$t$),
also $\zeta=R^{n,n+1}_x\circ \eta$
is differentiable at~$t$ and
\begin{eqnarray*}
\dot{\zeta}(t) &=&
TR^{n,n+1}_x\dot{\eta}(t)=
TR^{n,n+1}_xTR^{n+1,n+4}_{\eta(t)}\gamma(t)
= TR^{n,n+4}_{\eta(t)x}\gamma(t)\\
&=&
TR^{n,n+4}_{\zeta(t)}\gamma(t)
=TR^{n,n+3}_{\zeta(t)}T\lambda^{n+3}_{n+4}\gamma(t)
=f_n\hspace*{-.8mm}\left(\zeta(t),T_e\lambda^{n+3}_{n+4}\gamma(t)\right),
\end{eqnarray*}
using (\ref{triv-2}) for the third equality
and (\ref{triv-1}) for the fifth equality.
\end{numba}
\begin{numba}\label{global}
Let $n\in \N$.
If $\beta\in\cL^1([0,1],\cm_{n+4})$
is arbitrary,
there exists $N\in \N$ such that
\[
\|\beta_k\|_{\cL^1}\leq \ve_{n+1}
\]
holds for the $\cL^1$-function
\[
\beta_k\colon [0,1]\to\cm_{n+4},\quad
t\, \mto\, \frac{1}{N}\, \beta\left( \frac{k+t}{N}\right)
\]
for
all $k\in \{0,\ldots, N-1\}$ (see \cite[Lemma~5.26]{MEA}).
Let $\theta_k\colon [0,1]\to M_{n+1}$
be the Carath\'{e}odory solution to
\[
\dot{y}(t)=f_{n+1}(y(t),\beta_k(t)),\quad y(0)=e.
\]
Using the Chain Rule,
we see that $\zeta_k\colon [k/N,(k+1)/N]\to M_{n+1}$,
$t\mto \theta_k(Nt-k)$ is absolutely
continuous and solves
\[
\dot{y}(t)=f_{n+1}(y(t),\beta(t)),\quad y(k/N)=e.
\]
We define $\zeta\colon [0,1]\to M_n$
piecewise via $\zeta(t):=\zeta_0(t)$ if $t\in [0,1/N]$,
\[
\zeta(t):=\zeta_k(t)\hspace*{.2mm}\zeta_{k-1}(k/N)\cdots \zeta_1(2/N)\zeta_0(1/N)
\]
for $t\in [k/N,(k+1)/N]$ with $k\in \{1,\ldots, N-1\}$,
multiplying from right to left as in~\ref{multiple}.
Then $\zeta$ takes values in~$M_n$, by~\ref{multiple}.
Moreover, $\zeta\colon [0,1]\to M_n$
is absolutely continuous, $\zeta(0)=e$
and $\dot{\zeta}(t)=f_n(\zeta(t),T\lambda_{n+4}^{n+3}\beta(t))$
for almost all $t\in [0,1]$, by~\ref{disc-zeta}. Thus $\zeta$
solves
\[
\dot{y}(t)=f_n(y(t),T_e\lambda^{n+3}_{n+4}\beta(t)).
\]
\end{numba}
\begin{numba}
If $\gamma\in\cL^1([0,1],\cg)$
with $\cg:=T_eG=T_e(M_\infty)$,
then $\gamma_n:=T_e\lambda_{n+3}\circ\gamma\in \cL^1([0,1],\cm_{n+3})$
for each $n\in \N$.
By~\ref{global},
the initial value problem
\[
\dot{y}(t)=f_n(y(t),\gamma_n(t)),\quad y(0)=e
\]
has a unique Carath\'{e}odory solution
$\eta_n\colon [0,1]\to M_n$,
using that $\gamma_n(t)=T\lambda_{n+4}^{n+3}\circ\gamma_{n+1}$.
If $m> n$ in $\N$,
then $\lambda_m^n\circ \eta_m\colon [0,1]\to M_n$ is absolutely
continuous and $(\lambda_m^n\circ\eta_m)(0)=e$.
For each $t\in [0,1]$
such that $\dot{\eta}_m(t)$ exists and
equals $f_m(\eta_m(t),\gamma_m(t))$, 
also $(\lambda^n_m\circ\eta_m)^{{\boldsymbol \cdot}}(t)$
exists and equals
\begin{eqnarray*}
T\lambda_m^n\dot{\eta}_m(t)=T\lambda_m^n TR_{\eta_m(t)}^{m,m+3}\gamma_m(t)
= TR_{\lambda_m^n(\eta_m(t))}^{n,n+3}\gamma_n(t)
=f_n(\lambda_m^n\eta_m(t),\gamma_n(t)),
\end{eqnarray*}
exploiting that $\lambda_m^n\circ R_{\eta_m(t)}^{m,m+3}\circ \lambda_m
= R^{n,n+3}_{\lambda_m^n(\eta_m(t))}\circ \lambda_n$.
By uniqueness,
\[
\eta_n=\lambda_m^n\circ \eta_m.
\]
Lemma~\ref{PL-lemma}\,(c) provides a unique
absolutely continuous map $\eta\colon [0,1]\to M_\infty$
such that $\lambda_n\circ \eta=\eta_n$
for all $n\in \N$.
Then
\begin{equation}\label{hencein}
\eta_n([0,1])=\eta([0,1])\sub M_\infty\quad\mbox{for all $\,n\in\N$.}
\end{equation}
Let us apply~\ref{global}
with $n:=\ell$ and $\beta:=\gamma_{\ell+1}$.
Then $\eta_\ell$ equals $\zeta$ constructed in~\ref{global}.
Using~$N$ as in \ref{global},
we define $\alpha_k\colon [0,1]\to\cg$
via
\[
\alpha_k(t):=\frac{1}{N}\hspace*{.3mm}\gamma\hspace*{-.6mm}\left(\frac{k+t}{N}\right).
\]
Then $T\lambda_{\ell+1}\circ \alpha_k=\beta_k$
for $\beta_k$ as in \ref{global},
whence the solution $\theta_k\colon [0,1]\to M_{\ell+1}$
to $\dot{y}(t)=f_{\ell+1}(y(t),\beta_k(t))$,
$y(0)=e$ has image
\[
\theta_k([0,1])\sub M_\infty,
\]
as we can apply (\ref{hencein})
with $\ell+1$ and $\alpha_k$ in place of~$n$
and~$\gamma$.
On the other hand,
\[
\theta_k
=s_{\ell+1}(\beta_k)
\]
takes values in~$Q$ (see \ref{the-eps}).
Hence
\[
\theta_k([0,1])\, \sub \, M_\infty\cap Q\, \sub \, U\, \sub \, G.
\]
For each $t\in [0,1]$,
the element $\zeta(t)\in M_\ell$
is a product of function values
of $\zeta_k$ (and hence of $\theta_k$)
in~\ref{global}, all of which are in~$G$
by the preceding.
Since $G$ is closed under products,
we deduce that
\[
\eta(t)=\eta_\ell(t)=\zeta(t)\in G\quad
\mbox{for all $\, t\in[0,1]$.}
\]
For $n\in\N$, let $A_n\sub [0,1]$
be a Borel set of Lebesgue measure~$0$
such that $\dot{\eta}_n(t)$
exists for all $t\in [0,1]\setminus A_n$
and equals $f_n(\eta_n(t),\gamma_n(t))$.
Then $A:=\bigcup_{n\in\N}A_n$
is a Borel set of measure $0$.
For $g\in G$, let $R_g\colon G\to G$, $x\mto xg$
the smooth right multiplication with~$g$.
Then
\[
\lambda_n \circ R_g= R_g^{n,m}\circ \lambda_m
\]
for all $m> n$ in~$\N$.
For $t\in [0,1]\setminus A$,
let
\[
v:=TR_{\eta(t)}(\gamma(t)).
\]
Then $\lambda_n\circ \eta=\eta_n$
is differentiable at $t$ for each $n\in\N$,
with
\begin{eqnarray*}
(\lambda_n\circ \eta)^{\textstyle .}(t) & =& \dot{\eta}_n(t)
=TR_{\eta_n(t)}^{n,n+3}\gamma_n(t)
=TR_{\eta_n(t)}^{n,n+3}T\lambda_{n+3}\gamma(t)\\
& =& T\lambda_nTR_{\eta(t)}\gamma(t)=T\lambda_n(v).
\end{eqnarray*}
By Lemma~\ref{PL-lemma}\,(d),
$\dot{\eta}(t)$ exists and equals
$v=TR_{\eta(t)}\gamma(t)$.
Thus
\[
\dot{\eta}(t)=\gamma(t).\eta(t).
\]
We have shown that $\eta=\Evol([\gamma])$.
\end{numba}
\begin{numba}
Note that the linear map
\[
L^1([0,1],\cg)\to L^1([0,1],\cm_{\ell+4}),\quad [\gamma]\mto [T\lambda_{\ell+4}\circ\gamma]
\]
is continuous, whence
\[
\Omega_P:=\{[\gamma]\in L^1([0,1],\cg)\colon \|T\lambda_{\ell+4}\circ\gamma\|_{\cL^1}\leq\ve\}
\]
is a $0$-neighbourhood in $L^1([0,1],\cg)$.
By \ref{the-eps}, we have
\[
\Evol([\gamma])=s_{\ell+1}(T\lambda_{\ell+3}\circ\gamma)\in C([0,1],P)
\]
for each $[\gamma]\in \Omega_P$.
As sets of the form $C([0,1],P)$
form a basis of identity-neighbourhoods in $C([0,1],G)$,
we deduce that
\[
\Evol\colon L^1([0,1],\cg)\to C([0,1],G)
\]
is continuous at~$0$.
Hence $\Evol\colon L^1([0,1],\cg)\to C([0,1],G)$
is smooth (and $G$ an $L^1$-regular Fr\'{e}chet-Lie group),
by \cite[Theorem 1.6]{ANA}.
\end{numba}
This completes the proof of Proposition~\ref{technical}. $\,\square$
{\small{\bf Helge Gl\"{o}ckner}, Universit\"{a}t Paderborn, Warburger Str.\ 100,
33098 Paderborn,\linebreak
Germany; glockner@math.uni-paderborn.de\\[3mm]
{\bf Ali Suri}, Universit\"{a}t Paderborn, Warburger Str.\ 100,
33098 Paderborn, Germany; asuri@math.uni-paderborn.de}\vfill
\end{document}